\documentclass[11pt,bezier]{article}

\usepackage{tabularx,booktabs}
\usepackage{amsmath,amssymb,amsfonts, verbatim}
\usepackage{algpseudocode,algorithm,algorithmicx}
\newcolumntype{Y}{>{\centering\arraybackslash}X}

\newtheorem{prethm}{{\bf  Theorem}}

\newenvironment{thm}{\begin{prethm}{\hspace{-0.5
               em}{\bf .}}}{\end{prethm}}
\newtheorem{prepro}{{\bf  Theorem}}

\newtheorem{precor}{{\bf  Corollary}}

\newtheorem{prepos}{{\bf  Preposition}}

\newtheorem{preexample}{{\bf Example}}

\newtheorem{preconj}{{\bf  Conjecture}}

\newenvironment{conj}{\begin{preconj}{\hspace{-0.5
               em}{\bf .}}}{\end{preconj}}
\newtheorem{preremark}{{\bf  Remark}}

\newtheorem{prelem}{{\bf  Lemma}}

\newtheorem{preclaim}{{\bf  Claim}}

\newenvironment{claim}{\begin{preclaim}{\hspace{-0.5
               em}{\bf .}}}{\end{preclaim}}

\newtheorem{preproof}{{\bf  Proof.}}

\newenvironment{proof}[1]{\begin{preproof}{\rm
               #1}\hfill{$\Box$}}{\end{preproof}}

\long\def\/*#1*/{}

\title{\large \bf On the Maximum Order of Induced Paths \\ and Induced Forests in Regular Graphs
}

\author{{\normalsize
{\sc S. Akbari${}^{\mathsf{a}}$},\,
{\text{\sc A. Amanihamedani${}^{\mathsf{b}}$}},\,
  {\text{\sc S. Mousavi${}^{\mathsf{b}}$}},\,
  {\text{\sc H. Nikpey${}^{\mathsf{b, c}}$}},\,
  {\sc \text{S. Sheybani}${}^{\mathsf{b}}$}},\,
 \vspace{3mm}
\\{\footnotesize{${}^{\mathsf{a}}$\it Department of
Mathematical Sciences, Sharif University of Technology, Tehran,
Iran}}
{\footnotesize{}}\\{\footnotesize{${}^{\mathsf{b}}$\it Department of
Computer Engineering, Sharif University of Technology, Tehran,
Iran}}
{\footnotesize{}}\\{\footnotesize{${}^{\mathsf{c}}$\it Department of
Computer and Information Science, University of Pennsylvania, Philadelphia,
USA}}\thanks{{\it E-mail addresses}: $\mathsf{s\_akbari@sharif.edu}$, $\mathsf{aramani@ce.sharif.edu}$,
$\mathsf{smousavi@ce.sharif.edu}$, $\mathsf{hesam@seas.upenn.edu}$, $\mathsf{sheybani@ce.sharif.edu}$.}}

\date{}
\begin{document}

\maketitle

\begin{abstract}
%We study several parameters in graphs such as maximum order of induced path and maximum order of linear forest. We provide some bounds on these parameters which apply to regular graphs. In this paper, it is shown that if $G$ is an $r$-regular graph of order $n$, then the order of the largest induced linear forest is at least $\frac{2}{r+1}n$. Our proofs lead to a polynomial time algorithm for finding a large induced linear forest. Moreover, a Nordhaus-Gaddum type inequality is presented for the order of maximum induced forest in a graph.

Let $G$ be a graph and $a(G)$, $\textup{LIF}(G)$ denote the maximum orders of an induced forest and an induced linear forest of $G$, respectively. It is well-known that if $G$ is an $r$-regular graph of order $n$, then $a(G) \geq \frac{2}{r+1}n$. In this paper, we generalize this result by showing that $\textup{LIF}(G) \geq \frac{2}{r+1}n$. It was proved that for every graph $G$, $a(G) \geq \sum_{i=1}^{n}\frac{2}{d_i+1}$, where $d_1, \ldots, d_n$ is the degree sequence of $G$. Here, we conjecture that for every  graph  $G$ with $\delta(G) \geq  2$, $\textup{LIF}(G) \geq \sum_{i=1}^{n}\frac{2}{d_i+1}$.

\end{abstract}

%\begin{center}


    
%\end{center}

%%%%%%%%%%%%%%%%%%%%%%%%%%%%%%%%%%%%%%%%%%%%%%%%%%%%%%%%%%%%%%%%%%%%%%%%%%%%%%%%%%%%%%%%%%%%%%%%%%%%%%%%%%%%%%%%%%%%%%%%%%%%%%%%%%%%%%%%%%%
%%%%%%%%%%%%%%%%%%%%%%%%%%%%%%%%%%%%%%%%%%%%%%%%%%%%%%%%%%%%%%%%%%%%%%%%%%%%%%%%%%%%%%%%%%%%%%%%%%%%%%%%%%%%%%%%%%%%%%%%%%%%%%%%%%%%%%%%%%%
%%%%%%%%%%%%%%%%%%%%%%%%%%%%%%%%%%%%%%%%%%%%%%%%%%%%%%%%%%%%%%%%%%%%%%%%%%%%%%%%%%%%%%%%%%%%%%%%%%%%%%%%%%%%%%%%%%%%%%%%%%%%%%%%%%%%%%%%%%%
\/*
Comment :D
*/

\section{Introduction}
Let $G$ be a graph with vertex set $V(G)$ and edge set $E(G)$. The number of vertices and edges of $G$ are called \textit{order} and \textit{size} of $G$, respectively. In this paper, we just consider simple graphs. The complement of $G$ is denoted by $\overline{G}$. The degree of a vertex $v$ is denoted by $d_G(v)$. Let $\delta(G)$ be the minimum degree of $G$. For two positive integers $m$ and $n$ let $K_n$ and $K_{m,n}$ denote the complete graph of order $n$ and the complete bipartite graph with part sizes $m$ and $n$. Let $d_1, \ldots, d_n$ be the degree sequence of $G$ and $t(G) = \sum_{i=1}^{n}\frac{1}{d_i+1}$. For every $v \in V(G)$ and any subgraph $H$ of $G$, $d_{H}(v)$ indicates the number of neighbors of $v$ in $V(H)$. Let $a(G)$ denote the maximum order of an induced forest in $G$. Let $\textup{LIP}(G)$ denote the order of a longest induced path in $G$. If every connected component of a forest $F$ is a path, then $F$ is said to be a \textit{linear forest}. Let $\textup{LIF}(G)$ denote the order of the largest induced linear forest in $G$. 
The well-known Caro-Wei bound states that $G$ contains an independent set of size at least $t(G)$, see \cite{caro}, \cite{MURPHY}, and \cite{wei}. In \cite{shixu}, it is shown that for a connected graph $G$ of order $n$ and size $m$, $a(G) \geq \frac{8n-2m-2}{9}$. The right hand side of this inequality is negative for $r$-regular graphs with $r \geq 8$, which is trivial. In this paper, we provide a non-trivial lower bound for $a(G)$, where $G$ is a regular graph. Moreover, there are some lower bounds on the order of the largest induced linear forest in a triangle-free planar graph and an outerplanar graph \cite{dross}, \cite{pelsmajer}. If $\delta(G) \geq 1$, then it is shown that $a(G) \geq 2 t(G)$, see \cite{punnim}. Also, the author shows that for an $r$-regular graph $G$ of order $n$, $a(G) \geq \frac{2}{r+1}n$. Here, we generalize this result to linear induced forest and prove that $\textup{LIF}(G) \geq \frac{2}{r+1}n$ for a general $r$-regular graph. Furthermore, we show that every $r$-regular graph $G$ with $\textup{LIP}(G) = k$ has order at least $2k + \lceil -\frac{2(k-1)}{r}\rceil$. In addition, there is a graph for which this bound is sharp. Next, we present a Nordhaus-Gaddum type inequality for $a(G)$ and provide a family of examples where equality holds for them, see \cite{Nord}. Finally, we present some computational results on order of the largest induced forest in a graph and conjecture that for every graph with $\delta(G) \geq 2$, $\text{LIF}(G) \geq 2t(G)$.

\section{A Sharp Lower Bound for the Order of Largest \\ Induced Linear Forest in Regular Graphs}

In this section, we obtain a lower bound for $\textup{LIF}(G)$ in terms of $n$ and $r$, where $G$ is an $r$-regular graph of order $n$. Finally, we present a polynomial time algorithm to find an induced linear forest of the given lower bound order.

\begin{thm} \label{lower_LIF}
\textit{
Let
$G$
be an $r$-regular graph of order $n$. Then, $\textup{LIF}(G) \geq \frac{2}{r+1}n$ and this bound is sharp.
}
\end{thm}

\begin{proof}
{
If
$r=1$,
then
the assertion is trivial.
Now, suppose that $r \geq 2$.
%We first show that the bound is sharp. Let
%$G = K_n$
%for
%$n \geq 3$.
%$G$
%is
%$(n-1)-$
%regular and its largest induced linear forest contains no more than 2 vertices as any subset of 
%$V(G)$ 
%with at least 3 vertices induces a triangle and is no longer acyclic. So in this case,
%$LIF(G) = 2$.
%On the other hand, we have
%$\frac{2}{r+1}n = \frac{2}{n-1+1}n = 2$.
%Hence, the sharpness is proved.
Let $G_1 = G$. At the first step, let
$F_1$
be an induced linear forest with maximum order of $G_1$ whose size is minimum. At the $i$-th step, $i \geq 2$, let $G_i = G \setminus \bigcup_{j=1}^{i-1}V(F_j)$ and $F_i$ be an induced linear forest of maximum order of $G_i$ whose size is minimum. Continue this procedure until $V(G)$ is partitioned into $k$ induced linear forests, namely $F_1, \ldots, F_k$. For
$i = 1,\ldots,k$,
let
$|F_i| = c_in$, for some $c_i \in (0, 1)$.
We have
$c_1 + \cdots + c_k = 1$ and $c_1 \geq \cdots \geq c_k > 0$. Since $V(G)$ is partitioned into
$k$
linear induced forests, we have

\begin{equation} \label{eq:12}
|F_1| = c_1n \geq \frac{n}{k}.
\end{equation}
Now, we state the following claim.

\begin{claim} \label{claim1}
For every
$v \in V(F_i)$,
$i \geq 2$ and $j < i$,
$d_{F_j}(v) \geq 2$.
\end{claim}

To prove the claim by contradiction suppose that there exists a vertex
$v \in V(F_i)$
and a linear forest
$F_j$
with
$j < i$ such that
$d_{F_j}(v) \leq 1$.
If
$d_{F_j}(v) = 0$, then by adding 
$v$
to
$F_j$
we obtain a larger induced linear forest which is a contradiction. Now, assume that
$v$ has exactly one neighbor in
$V(F_j)$, say $u$. If $u$
is an end vertex of an induced path, then by adding $v$ to $F_j$ we obtain a larger induced linear forest, a contradiction. If $u$ is a non-leaf vertex of an induced path, then by adding $v$ to $F_j$ and removing $u$ we obtain an induced linear forest of the same order but with two less edges, a contradiction and the claim is proved.

To continue the proof, we consider two different cases depending on the parity of
$r$.

\textbf{Case 1.} $2 \nmid r$.

Consider a vertex
$v \in V(F_k)$.
By the Claim \ref{claim1}, we know that
$v$
has at least
$2(k-1)$
neighbors outside of
$F_k$.
Since
$G$
is
$r$-regular and 
$r$
is odd, we have
$2(k-1) \leq r-1$
which implies that 
$k \leq \frac{r+1}{2}$.
By (\ref{eq:12}), we have
\begin{equation} \label{eq:14}
|F_1| \geq \frac{2}{r+1}n.
\end{equation}
Hence, the assertion is proved in this case.

\textbf{Case 2.} $2 \mid r$.

Consider a vertex
$v \in V(F_k)$.
By the Claim \ref{claim1}, we know that
$v$
has at least
$2(k-1)$
neighbors outside of
$F_k$.
Since
$G$
is
$r$-regular,
we have
$2(k-1) \leq r$. 
If $2(k-1) \leq r-1$, then the proof is similar to the previous case.
Now, if
$2(k-1) = r$, then noting that every vertex
$v \in F_k$
has
$r$
neighbors outside of
$F_k$, 
$F_k$
is an independent set. For completing the proof, we need the following claim.

\begin{claim} \label{claim2}
If
$2(k-1) = r$, then
every vertex of
$F_k$
has exactly two pendant neighbors in
$V(F_{k-1})$ and $c_{k-1} \geq 2c_k$.
\end{claim}

By the Claim \ref{claim1} and the equality
$2(k-1) = r$,
we conclude that every vertex in
$F_k$
has exactly two neighbors in each
$F_i$,
$i < k$.
Also, a vertex of degree 2 in $F_{k-1}$ has at least $2(k-2) = r-2$ neighbors in $\bigcup_{i=1}^{k-2}F_i$. Along with two neighbors in $F_{k-1}$, there remains no edges between the vertex and vertices in $F_k$. Hence, each edge between $F_k$ and $F_{k-1}$ is not adjacent to a vertex of degree 2 in $F_{k-1}$.

Now, consider a vertex $v \in V(F_{k})$ and suppose that its two neighbors in $V(F_{k-1})$ are $u$ and $w$. Both $u$ and $w$ are end vertices of two paths, say $P_u$ and $P_w$, respectively. Suppose that $P_u \neq P_w$. By adding $v$ to $F_{k-1}$ we obtain a larger linear induced forest, a contradiction. Thus, we have $P_u = P_w$. Note that no other vertex in $F_k$ can be adjacent to $u$ (similarly to $w$), because the degree of $u$ exceeds $r$. Therefore, every vertex in $F_k$ is adjacent to the endpoints of a path in $F_{k-1}$ and these paths are distinct. Hence, we have $|F_{k-1}| \geq 2|F_{k}|$ and the claim is proved.

In the sequel, if $c_k \geq \frac{1}{r+1}$, then by Claim \ref{claim2} we have $c_1 \geq c_{k-1} \geq 2c_k \geq \frac{2}{r+1}$, as desired. If $c_k < \frac{1}{r+1}$, then let
$G'$
be the induced subgraph on
$V(G) \setminus V(F_k)$.
We have
$|V(G')| = n(1-c_k)$.
Note that $V(G')$
is partitioned into
$k-1$
linear forests, namely
$F_1, \ldots, F_{k-1}$. Thus, we have

$$|F_1| = c_1n \geq \frac{n(1-c_k)}{k-1}.
$$
Since 
$2(k-1) = r$,
one can see that
$$
|F_1| = c_1n \geq \frac{2}{r+1}n,
$$
and the assertion is proved. Note that for every positive integer $n \geq 2$, $\text{LIF}(K_n) = 2$.
}
\end{proof}

In the following, we present an algorithm for finding an induced linear forest in $G$ of order at least $\frac{2}{r + 1}n$. We aim to construct a sequence of subgraphs $\{G_i\}$ in which $G_i$ is a maximal induced linear forest in $G \setminus \bigcup_{j=1}^{i-1}G_j$. Algorithm \ref{alg} describes this approach.

\begin{algorithm}[H]
  \caption{Finding induced linear forest of order at least $\frac{2}{r + 1}n$
    \label{alg}}
  \begin{algorithmic}[1]
    \vspace{0.25cm}
    \Statex
      Let $G'$ be initially equal to $G$ and set $j = 0$. Additionally, let all $G_i$ be empty subgraphs.
      \While{$G'$ is not empty}
      \State
      $j = j + 1$
        \While{there exists $v \in G'$ s.t. $G_j \cup \{v\}$ is an induced linear forest}
            \State
            Remove $v$ from $G'$ and add it to $G_j$
            \While{there exists $w \in G'$ s.t. $w$ has exactly one neighbor in $G_j$ called $t$ and $d_{G_j}(t) = 2$}
                \State
                Remove $t$ from $G_j$ and add it to $G'$
                \State
                Remove $w$ from $G'$ and add it to $G_j$
        \EndWhile
        \EndWhile
        \If{there is a $k < j$ where $|V(G_k)| < |V(G_j)|$}
            \State
            $k_0 = \underset{k}{\arg\min} \{ k \, | \, |V(G_k)| < |V(G_j)|\}$
            \State
            Swap $G_{k_0}$ and $G_j$
            \State
            Add $G_{k_0 + 1},\ldots,G_j$ to $G'$ and set each of them to empty subgraph again
            \State
            $j = k_0$
        \EndIf
      \EndWhile
  \end{algorithmic}
\end{algorithm}

There are several considerations about this algorithm. First, it should be proved that this algorithm terminates. Obviously, the \textit{while} loop in Line 1 cannot be executed infinitely. Furthermore, the condition of \textit{while} in the Lines 3 and 5 cannot be true more than $n$ times. Thus, the total number of operations is finite and the algorithm terminates. Second, because of the swap operation inside the \textit{if} in Line 10, the orders of subgraphs in the final sequence after running the algorithm are non-increasing. Because of the analogy between this construction and the one in the proof, Claim \ref{claim1} and Claim \ref{claim2} holds for this sequence as well. Therefore, we can conclude that using maximal induced linear forests rather than maximum orders, we can reach an induced linear forest $G_1$ of order at least $\frac{2}{r + 1}n$. Additionally, since each \textit{while} loop would take at most $n$ iterations, the run-time of the algorithm is $O(n^3)$. Although this algorithm does not find the largest induced linear forest of $G$, it finds one with the given lower bound in polynomial time.

\begin{comment}

Maximal linear forests would work as maximums with minimum number of edges: in each step $G_i$ is the induced linear forest of the remaining graph with:

1) No vertex outside of $G_i$ can be added to $G_i$.

2) No vertex outside of $G_i$ has exactly one edge to $G_i$, if so, add the vertex to $G_i$ and remove the adjacent vertex from $G_i$.

Also, in every step if $|V(G_i)| > |V(G_j)|$ and $i > j$, replace $G_i$ with $G_j$ and continue from Step $j$. Obviously, this procedure would be finished in polynomial time and has the properties of the global maximum which we used. Therefore, the bound can be achieved by the polynomial time algorithm.

\end{comment}

\begin{comment}
\begin{conj} \label{conjecture}
\textit{
For a graph $G$ with degree sequence $\textbf{d}(G) = (d_1, d_2, \ldots, d_n)$, $\textup{LIF}(G) \geq 2\sum_{i=1}^{n}\frac{1}{d_i+1}$
}
\end{conj}

This conjecture has been tested on graphs with at most 9 vertices and holds for all of them. In addition, experiments on random graphs with higher orders are in accordance with our guess. 
\end{comment}

\section{Sharp Upper Bound for the Order of the Longest\\ Induced Path in Regular Graphs}

In this section, we present a lower bound on the order of an $r$-regular graph $G$ with $\textup{LIP}(G) = k$ in terms of $r$ and $k$.

\begin{thm} \label{MIP}
	For every $r$-regular graph $G$ of order $n$ with $\textup{LIP}(G) = k$, the following holds:
	$$
    	n \geq 2k + \lceil \frac{-2(k-1)}{r} \rceil.
	$$
	 
\end{thm}

\begin{proof}
	{
	Let $P$ be a longest induced path in $G$. Each non-leaf vertex in $P$ has exactly $r-2$ neighbors outside of the path and each of the end vertices of $P$ has $r-1$ neighbors outside of $P$. Thus, the total number of edges between $P$ and $G\setminus V(P)$ is $2(r-1) + (k-2) (r-2) = rk - 2(k-1)$. On the other hand, each vertex of $G\setminus V(P)$ is adjacent to at most $r$ vertices of $P$. Hence, there are at least ${\frac{rk - 2(k-1)}{r}} = k - {\frac{2(k-1)}{r}}$ vertices in $G\setminus V(P)$. Therefore, $G$ has order at least $2k + \lceil \frac{-2(k-1)}{r} \rceil$. Since the sum of all degrees in $G$ is even, if $r$ and $\lceil \frac{-2(k-1)}{r} \rceil$ are both odd, then the order of $G$ is at least $2k + \lceil \frac{-2(k-1)}{r} \rceil + 1$. 
	
	We show that there are some graphs which meet the upper bound. Now, for every positive integer $r$, we construct an $r$-regular graph $G$ of order $n$ with $\textup{LIP}(G) = r$, where the equality holds. Consider a path of order $r$, say $P$, and a set of vertices $S$ outside of $P$. If $r$ is even, let $S$ consist of $r-1$ vertices. Otherwise, let $S$ contain $r$ vertices. In the case of even $r$, pick a vertex in $S$, say $v$, and join each of the other $r-2$ vertices of $S$ to all vertices of $P$. Afterward, remove the edges of an arbitrary matching of size $\frac{r-2}{2}$ between the non-end vertices of $P$ and $S \setminus \{v \} $. Now, join $v$ to all vertices of the removed matching.

	%this paragraph should be revised
	
	If $r$ is odd, then pick $u$ and $w$ from $S$ and join all other $r-2$ vertices of $S$ to all vertices of $P$ to obtain $K_{r,r-2}$. now remove a matching of size $r - 2$ of this $K_{r,r-2}$. Next, join $w$ to $u$ and the $r-2$ vertices in $S$. Then, join $u$ to all non-end vertices in $P$. Finally, join $u$ to one of the end vertices of $P$ and join $w$ to the other one. One can check that these graphs are $r$-regular and have a longest induced path of order $r$. Therefore, the equality holds for this family of graphs.
}
\end{proof}

Now, we have an immediate corollary.

\begin{precor}
	In any $r$-regular graph of order $n$ with $\textup{LIP}(G) = k$, $k \leq \frac{rn - 2}{2r - 2}$.
\end{precor}

\textbf{Note on Complexity.} In \cite{np}, it is shown that determining whether the graph $G$ has an induced path of length at least $k$, with arbitrary $k$, is NP-complete. Consequently, finding the LIP of the graph is NP-complete, too. However, while this theorem holds for general graphs, finding the longest induced path, to the best of our knowledge, is neither proved to be NP-complete, nor any polynomial time algorithm has been presented for, up to now.

\section{Nordhaus-Gaddum Inequality for the Order of \\ Largest Induced Forest}

In this section, we present a Nordhaus-Gaddum type inequality for the order of maximum induced forest and present a family of examples which the equality case holds for them.

\begin{thm} \label{upperbound_sum} 
\textit{
Let 
$G$
be a graph of order $n$. Then, $a(G) + a(\overline{G}) \leq n+4$ and this bound is sharp.
}
\end{thm}

\begin{proof}
{
	Let $F_1$ and $F_2$ be the largest induced forests of $G$, and $\overline{G}$, respectively. By contradiction, assume that $a(G) + a(\overline{G}) \geq n+5$. Since $|V(F_1)\cup V(F_2)| \leq n$, we have $|V(F_1)\cap V(F_2)| \geq 5$. Let $X = V(F_1)\cap V(F_2)$ and $|X| = t$. With no loss of generality and using pigeonhole principle, the induced subgraph on $X$ in $F_1$ has at least $\frac{\binom{t}{2}}{2} = \frac{t(t-1)}{4}$ edges. Since $t \geq 5$, $\frac{t(t-1)}{4} \geq t$ and the induced subgraph on $X$ in $F_1$ contains a cycle we obtain a contradiction. Hence, $a(G) + a(\overline{G}) \leq n+4$, for any graph of order $n$. To prove the sharpness of the bound, we present a family of graphs for which the equality holds. Consider $P_n$ for $n \geq 4$. The path $P_n$ is an induced forest and therefore, $a(G) = n$. Clearly, $\overline{P_n}$ has an induced subgraph $P_4$. Thus, we have $a(\overline{G}) \geq 4$ and the proof is complete.
}
\end{proof}

Theorem \ref{upperbound_sum} has an immediate corollary.

\begin{precor}
Let $G$ be a graph of order $n$. Then, $\textup{LIF}(G) + \textup{LIF}(\overline{G}) \leq n+4$. Moreover, this bound is sharp.
\end{precor}

\section{Computational Results and One Conjecture}

In this section, we present some computational results and propose a conjecture. We find the maximum and minimum order of largest induced forest in addition to largest and smallest LIF among all cubic graphs with fixed orders, up to 18 vertices. Graphs data were produced using GENREG \cite{GENREG} and brute-force search is carried out to find the statistics. Summary of our results can be found in Table \ref{tab:comp_results}.

\begin{table}[h]%
\caption{Computational Results for Cubic Graphs \label{tab:comp_results}}

\begin{tabularx}{\textwidth}{|c *{8}{Y}|}
%\toprule
\hline
\hspace*{1cm}
 %& \multicolumn{7}{c}{Number of Vertices}\\
 %\cmidrule(1){2-9}
  $|V(G)|$  & 4    & 6      & 8 & 10 & 12 & 14 & 16 & 18 \\ 
  \hline
%\midrule
Maximum order of largest induced forest          & 2    & 4      & 5 & 7  & 8 & 10 & 11 & 13 \\
Minimum order of largest induced forest          & 2    & 4      & 5 & 6  & 8 & 9  & 10 & 11 \\
Maximum order of LIF   & 2    & 4      & 5 & 7  & 8 & 10 & 11 & 13 \\
Minimum order of LIF   & 2    & 3      & 5 & 6  & 7 & 8  & 10 & 11 \\
%\bottomrule
\hline
\end{tabularx}
\end{table}%

Furthermore, Punnim showed that $a(G) \geq 2t(G)$ \cite{punnim}. Let $f(G) = \frac{\text{LIF}(G)}{t(G)}$. By Caro-Wei Theorem \cite{MURPHY}, we have $f(G) \geq 1$. Let $g(n) = \min\{f(G) \, | \, |V(G)| = n\}$, where $G$ is connected. We have calculated $g(n)$ for $n$ up to 10. Table \ref{tab:carowei} shows our computational results.

\begin{table}[h]%
\caption{Computational Results for $g(n)$ \label{tab:carowei}}

\begin{tabularx}{\textwidth}{|c *{8}{Y}|}

%\toprule
\hline
\hspace*{1cm}
 %& \multicolumn{7}{c}{Number of Vertices}\\
 %\cmidrule(l){2-8}
  $|V(G)|$  & 3 & 4 & 5 & 6 & 7 & 8 & 9 & 10 \\ 
%\midrule
\hline
\hspace{1cm}$g(n)$  & 2 & $\frac{12}{7}$ & $\frac{20}{11}$ & $\frac{15}{8}$ & $\frac{75}{41}$ & $\frac{30}{17}$ & $\frac{105}{58}$ & $\frac{35}{19}$ \\
\hline
%\bottomrule
\end{tabularx}

\end{table}

These computations can help us to find bounds for these measures and attack them theoretically. Moreover, by performing an exhaustive search on all graphs of order at most 10 and minimum degree at least 2, we propose the following conjecture.

\begin{conj}
If $G$ is a graph with $\delta(G) \geq 2$, then $f(G) \geq 2$.
\end{conj}

$K_{2,3}$, $K_{3,3}$, and the complete graphs are the only graphs of order at most 10 and minimum degree at least 2 that attain the equality.
\\


\begin{thebibliography}{mm}

\bibitem{caro} Y. Caro, New results on the independence number, Technical Report, Tel-Aviv University (1979).

\bibitem{dross}F. Dross, M. Montassier, and A. Pinlou, A lower bound on the order of the largest induced linear forest in triangle-free planar graphs, Discrete Mathematics 342, no. 4 (2019): 943-950.

\bibitem{np} M.R. Garey, and D.S. Johnson, Computers and Intractability: A Guide to the Theory of NP-completeness, Freeman, Fundamental (1997).

\bibitem{GENREG} M. Meringer, Fast Generation of Regular Graphs and Construction of Cages, Journal of Graph Theory 30 (1999): 137-146.

\bibitem{MURPHY} O. Murphy, Lower bounds on the stability number of graphs computed in terms of degrees, Discrete Mathematics 90, no. 2 (1991): 207-211.

\bibitem{Nord} E.A. Nordhaus, and J.W. Gaddum, On complementary graphs, The American Mathematical Monthly 63.3 (1956): 175-177.

\bibitem{pelsmajer} M.J. Pelsmajer, Maximum induced linear forests in outerplanar graphs, Graphs and Combinatorics 20.1 (2004): 121-129.

\bibitem{punnim} N. Punnim, Forests in random graphs, Southeast Asian Bulletin of Mathematics 27.2 (2003).

\bibitem{shixu} L. Shi, and H. Xu, Large induced forests in graphs, Journal of Graph Theory 85, no. 4 (2017): 759-779.

\bibitem{wei} V.K. Wei, A lower bound on the stability number of a simple graph, Bell Laboratories Technical Memorandum 81-11217-9, Murray Hill, NJ (1981).


\end{thebibliography}
\end{document}